\documentclass[twoside,12pt]{article}
 \usepackage{bbm}
 \usepackage{mathrsfs}
  \usepackage{amsfonts}
\usepackage{amsmath}

\newtheorem{theorem}{Theorem}[section]%
\newtheorem{lemma}[theorem]{Lemma}%
\newtheorem{conjecture}[theorem]{Conjecture}

\newcommand{\abs}[1]{\lvert#1\rvert}

\newenvironment{proof}[1][Proof]{\noindent\textit{#1: } }{\hfill\rule{1mm}{2mm}}

\makeatletter \@addtoreset{equation}{section} \makeatother

\begin{document}

  \title{A note on directed 4-cycles in digraphs
 \thanks{Supported by the Key Project of Chinese Ministry of Education (109140) and NNSF of China (No. 11071233).}
  }
 \author{Hao Liang\thanks{Corresponding author: lianghao@mail.ustc.edu.cn}\,\,
 \\
  {\small Department of Mathematics}\\
  {\small Southwestern University of Finance and Economics}\\
  {\small Chengdu {\rm 611130}, China}\\
  \\
  Jun-Ming Xu \\
  {\small School of Mathematical Sciences}\\
  {\small University of Science and Technology of China}\\
 {\small Wentsun Wu Key Laboratory of CAS}\\
  {\small Hefei {\rm 230026}, China}\\
 }
\date{}

\maketitle {\centerline{\bf\sc Abstract}\vskip 8pt  Using some
combinatorial techniques, in this note, it is proved that if
$\alpha\geq 0.28866$, then any digraph on $n$ vertices with
minimum outdegree at least $\alpha n$ contains a directed cycle of
length at most 4.



\section{Introduction}

Let $G=(V,E)$ be a digragh without loops or parallel edges, where $V=V(G)$
is the vertex-set and $E=E(G)$ is the arc-set. In 1978, Caccetta and
H\"{a}ggkvist~\cite{ch78} made the following conjecture:

\begin{conjecture}
Any digraph on n vertices with minimum outdegree at least r contains
a directed cycle of length at most $\lceil n/r\rceil$.
\end{conjecture}

Trivially, this conjecture is true for $r=1$, and it has been proved
for $r=2$ by Caccetta and H\"aggkvist~\cite{ch78}, $r=3$ by
Hamildoune~\cite{h87}, $r=4$ and $r=5$ by Ho\'ang and
Reed~\cite{hr87}, $r<\sqrt{n/2}$ by Shen~\cite{sj00}. While the
general conjecture is still open, some weaker statements have been
obtained. A summary of results and problems related to the
Caccetta-H\"aggkvist conjecture sees Sullivan~\cite{s00}.

For the conjecture, the case $r=n/2$ is trivial, the case $r=n/3$
has received much attention, but this special case is still open. To
prove the conjecture, one may seek as small a constant $\alpha$ as
possible such that any digraph on $n$ vertices with minimum
outdegree at least $\alpha n$ contains a directed triangle. The
conjecture is that $\alpha=1/3$. Caccetta and
H\"{a}ggkvist~\cite{ch78} obtained $\alpha\leq (3-\sqrt{5})/2\approx
0.3819$, Bondy~\cite{b97} showed $\alpha\leq (2\sqrt{6}-3)/5\approx
0.3797$, Shen~\cite{s98} gave $\alpha\leq 3-\sqrt{7}\approx 0.3542$,
Hamburger, Haxell, and Kostochka~\cite{hhk07} further improved this
bound to $0.35312$ by using a result of Chudnovsky, Seymour and Sullivan~\cite{css08}. Namely, any digraph on $n$ vertices with
minimum outdegree at least $0.35312n$ contains a directed triangle.

In this note, we consider the
minimum constant $\alpha$ such that
any digraph on $n$ vertices with minimum outdegree at least $\alpha
n$ contains a directed cycle of length at most 4. The
conjecture is that $\alpha=1/4$. Applying the combinatorial techniques
in~\cite{b97, lb83, s98}, we prove the following result.

\begin{theorem}\label{thm1.2}
If $\alpha \geq 0.28866$, then any digraph on $n$ vertices with
minimum outdegree at least $\alpha n$ contains a directed cycle of
length at most $4$.
\end{theorem}

\section{Proof of Theorem~\ref{thm1.2}}

We prove Theorem~\ref{thm1.2} by induction on $n\geq 3$. The theorem
holds for $n\leq 4$ clearly. Now assume that the theorem holds for
all digraphs with fewer than $n$ vertices. Let $G$ be a digraph on $n$ vertices with minimum
outdegree at least $\alpha n$. Suppose $G$ contains no directed
cycles with length at most 4. We can, without loss of generality,
suppose that $G$ is $r$-outregular, where $r=\lceil \alpha n
\rceil$, that is, every vertex is of the outdegree $r$ in $G$. We
will try to deduce a contradiction. First we present some
notations following~\cite{s98}.

For any $v\in V(G)$,
let

$N^+(v)=\{u\in V(G):(v,u)\in E(G)\}$, and $deg^+(v)=|N^+(v)|$, the
outdegree of $v$;

$N^-(v)=\{u\in V(G):(u,v)\in E(G)\}$, and $deg^-(v)=|N^-(v)|$, the
indegree of $v$.

We say $\langle u,v,w\rangle$ a {\it transitive triangle} if $(u,v),
(v,w), (u,w)\in E(G)$. The arc $(u,v)$ is called the base of the
transitive triangle.

For any $(u,v)\in E(G)$, let

$P(u,v)=N^+(v)\backslash N^+(u)$, and $p(u,v)=\abs {N^+(v)\backslash
N^+(u)}$, the number of induced 2-path with the first arc $(u,v)$;

$Q(u,v)=N^-(u)\backslash N^-(v)$, and $q(u,v)=\abs {N^-(u)\backslash
N^-(v)}$, the number of induced 2-path with the last arc $(u,v)$;

$T(u,v)=N^+(u)\cap N^+(v)$, and $t(u,v)=\abs {N^+(u)\cap N^+(v)}$,
the number of transitive triangles with base $(u,v)$.

\begin{lemma}
For any $(u,v)\in E(G)$,
 \begin{equation}\label{e2.1}
  n> r+deg^-(v)+q(u,v)+(1-\alpha)r+(1-\alpha)^2t(u,v).
  \end{equation}
\end{lemma}

\begin{proof} If $t(u,v)=0$, then the inequality (\ref{e2.1}) is
\begin{equation}\label{e2.2}
 n> r+deg^-(v)+q(u,v)+(1-\alpha) r.
 \end{equation}

There exists some $w\in N^+(v)$ with outdegree less than $\alpha r$
in the subdigraph of $G$ induced by $N^+(v)$ (Otherwise, this
subdigraph would contain a directed 4-cycle by the inductive
hypothesis). Thus $\abs {N^+(w)\backslash N^+(v)}\geq r-\alpha r$.
It is easy to see that four subsets $N^+(v)$, $N^+(w)\backslash
N^+(v)$, $N^-(v)$ and $N^-(u)\backslash N^-(v)$ are
pairwise-disjoint. It follow that
 $$
 \begin{array}{rl}
 n &>|N^+(v)|+|N^-(v)|+|N^-(u)\backslash N^-(v)|+|N^+(w)\backslash N^+(v)|\\
 &\geq r+deg^-(v)+q(u,v)+(1-\alpha) r.
 \end{array}
 $$
Thus, the inequality (\ref{e2.2}) holds for $t(u,v)=0$.

We now assume $t(u,v)>0$. By the inductive hypothesis, some vertex
$w\in N^+(u)\cap N^+(v)$ has outdegree less than $\alpha t(u,v)$ in
the subdigraph of $G$ induced by $N^+(u)\cap N^+(v)$, otherwise,
this subdigraph would contain a directed 4-cycle. Also, $w$ has not
more than $p(u,v)$ outneighbors in the subdigraph of $G$ induced by
$N^+(v)\backslash N^+(u)$. Let $N^+(w)\backslash N^+(v)$ be the
outneighbors of $w$ which is not in $N^+(v)$. Noting that
$t(u,v)=r-p(u,v)$, we have that
 \begin{equation}\label{e2.3}
 \abs {N^+(w)\backslash N^+(v)}\geq r-p(u,v)-\alpha
 t(u,v)=(1-\alpha)t(u,v).
 \end{equation}

Because $G$ has no directed triangle, these vertices are neither in
$N^-(v)$ nor in $N^-(u)\backslash N^-(v)$. Consider the subdigraph
of $G$ induced by $N^+(v)\cup N^+(w)$, by the inductive hypothesis,
some vertex $x\in N^+(v)\cup N^+(w)$ has outdegree less than $\alpha
\abs {N^+(v)\cup N^+(w)}$ in this subdigraph. Thus, the set of
outneighbors of $x$ not in $N^+(v)\cup N^+(w)$ satisfies
 $$
 \begin{array}{ll}
 \abs {N^+(x)\backslash (N^+(v)\cup N^+(w))}& \geq r-\alpha \abs {N^+(v)\cup N^+(w)}\\
                                       &=r-\alpha(\abs {N^+(v)}+\abs {N^+(w)\backslash N^+(v)})\\
                                       &=(1-\alpha)r-\alpha \abs {N^+(w)\backslash N^+(v)},
\end{array}
 $$
that is
\begin{equation}\label{e2.4}
  \abs {N^+(x)\backslash (N^+(v)\cup N^+(w))}\geq (1-\alpha)r-\alpha \abs {N^+(w)\backslash N^+(v)}.
\end{equation}

Since $G$ has no directed 4-cycle, these vertices are neither in
$N^-(v)$ nor in $N^-(u)\backslash N^-(v)$. Then because $N^+(v)$,
$N^+(w)\backslash N^+(v)$, $N^+(x)\backslash (N^+(v)\cup N^+(w))$,
$N^-(v)$ and $N^-(u)\backslash N^-(v)$ are pairwise-disjoint sets of
cardinalities $r$, $\abs {N^+(w)\backslash N^+(v)}$, $\abs
{N^+(x)\backslash (N^+(v)\cup N^+(w))}$, $deg^-(v)$ and $q(u,v)$, we
have that
 \begin{equation}\label{e2.5}
 n>r+\abs {N^+(w)\backslash N^+(v)}+\abs {N^+(x)\backslash (N^+(v)\cup N^+(w))}+deg^-(v)+q(u,v).
 \end{equation}

Substituting inequalities (\ref{e2.3}) and (\ref{e2.4}) into
(\ref{e2.5}) yields
 $$
 \begin{array}{ll}
 n& >r+\abs {N^+(w)\backslash N^+(v)}+(1-\alpha)r-\alpha \abs {N^+(w)\backslash N^+(v)}+deg^-(v)+q(u,v)\\
   &=r+(1-\alpha)r+(1-\alpha)\abs {N^+(w)\backslash N^+(v)}+deg^-(v)+q(u,v)\\
   &\geq  r+deg^-(v)+q(u,v)+(1-\alpha)r+(1-\alpha)^2t(u,v).
 \end{array}
 $$

The lemma follows.
\end{proof}

We now prove Theorem~\ref{thm1.2}. Recalling that $t(u,v)=r-p(u,v)$,
we can rewrite the inequality (\ref{e2.1}) as
 \begin{equation}\label{e2.6}
 (2\alpha-\alpha^2)t(u,v)>(3-\alpha)r-n+deg^-(v)+q(u,v)-p(u,v).
 \end{equation}

Summing over all $(u,v)\in E(G)$, we have that
 \begin{equation}\label{e2.7}
 \sum\limits_{(u,v)\in E(G)} t(u,v)=t,
 \end{equation}
where $t$ is the number of transitive triangles in $G$, and
 \begin{equation}\label{e2.8}
 \sum\limits_{(u,v)\in E(G)} (3-\alpha)r-n=nr[(3-\alpha)r-n].
 \end{equation}
By Cauchy's inequality and the first theorem on graph theory (see,
for example, Theorem 1.1 in~\cite{x03}), we have that
 $$
 \sum\limits_{(u,v)\in E(G)} deg^-(v)=\sum\limits_{v\in V(G)}
 (deg^-(v))^2\geq \frac{1}{n}\left(\sum\limits_{v\in V(G)}
 deg^-(v)\right)^2=nr^2,
 $$
that is
\begin{equation}\label{e2.9}
 \sum\limits_{(u,v)\in E(G)} deg^-(v)\geq nr^2.
 \end{equation}

Because $\sum\limits_{(u,v)\in E(G)} p(u,v)$ and $\sum\limits_{(u,v)\in E(G)} q(u,v)$
are both equal to the number of induced directed 2-paths in $G$, it follows that
 \begin{equation}\label{e2.10}
 \sum\limits_{(u,v)\in E(G)} p(u,v)=\sum\limits_{(u,v)\in E(G)} q(u,v).
 \end{equation}

Summing over all $(u,v)\in E(G)$ for the inequality (\ref{e2.6}) and
substituting inequalities (\ref{e2.7}) $\sim$ (\ref{e2.10}) into
that inequality yields,
 \begin{equation}\label{e2.11}
 (2\alpha-\alpha^2)t>(4-\alpha)nr^2-n^2r.
 \end{equation}

Noting that $t\leq n{r\choose 2}$ (see Shen~\cite{s98}), we have
that
  \begin{equation}\label{e2.12}
 t(2\alpha-\alpha^2)\leq
 nC_r^2(2\alpha-\alpha^2)<\frac{nr^2}{2}(2\alpha-\alpha^2).
 \end{equation}
Combining (\ref{e2.11}) and (\ref{e2.12}) yields
 \begin{equation}\label{e2.13}
 (4-\alpha)nr^2-n^2r<\frac{nr^2}{2}(2\alpha-\alpha^2).
 \end{equation}

Dividing both sides of the inequality (\ref{e2.13}) by
$\frac{nr^2}{2}$, and noting that $r=\lceil \alpha n \rceil$, we get
$$2(4-\alpha)-\frac{2}{\alpha}<(2\alpha-\alpha^2),$$
that is
$$\alpha^3-4\alpha^2+8\alpha-2<0.$$

We obtain that $\alpha<0.28865$, a contradiction. This completes the
proof of the theorem.

\end{document}